\documentclass[12pt, twoside, reqno]{amsart}
\usepackage{amscd}
\usepackage{bbm}
\usepackage{dsfont}
\usepackage{enumerate}
\usepackage{latexsym}
\usepackage{srcltx}
\usepackage{amsfonts}
\usepackage{amssymb}
\usepackage{datetime}
\usepackage{setspace}
\usepackage{enumerate}
\usepackage{amsthm}
\usepackage{graphicx}
\usepackage{epstopdf}

\newtheorem{theorem}{Theorem}[section]
\newtheorem{proposition}[theorem]{Proposition}
\newtheorem{lemma}[theorem]{Lemma}
\newtheorem{corollary}[theorem]{Corollary}

\theoremstyle{definition}

\newtheorem{definition}[theorem]{Definition}

\begin{document}
\title{ Spectrum of Weighted Composition Operators \\
Part VIII \\
Lower semi-Fredholm spectrum of weighted composition operators
on $C(K)$. The case of non-invertible surjections.}

\author{Arkady Kitover}

\address{Community College of Philadelphia, 1700 Spring Garden St., Philadelphia, PA, USA}

\email{akitover@ccp.edu}

\author{Mehmet Orhon}

\address{University of New Hampshire, 105 Main Street
Durham, NH 03824}

\email{mo@unh.edu}


\date{\today}


\maketitle

\markboth{Arkady Kitover and Mehmet Orhon}{Spectrum of weighted composition operators. VIII}

\section{Introduction} In this paper we continue the study of essential spectra of weighted composition operators. In~\cite{Ki3} the first named author obtained a description of essential spectra of weighted composition operators on $C(K)$ in the case when the corresponding map is a homeomorphism of the compact space $K$ onto itself. That allowed to provide a similar description of essential spectra of operators of the form $wU$ acting on Banach lattices in the case when $w$ is a central operator, $U$ is a $d$-isomorphism, and the spectrum of $U$ lies on the unit circle.

In~\cite{KO2} the authors described essential spectra of weighted composition operators on $C(K)$ in the case when the corresponding map is a non-invertible homeomorphism of $K$ into itself.

In~\cite{KO1} we obtained some general results concerning essential spectra of $d$-homomorphisms of Banach $C(K)$-modules. From this results follows as a very special case that the upper semi-Fredholm spectrum of a weighted composition operator on $C(K)$ is rotation invariant provided that the set of eventually periodic points of the corresponding map is of first category in $K$.

Despite these partial results, the problem of describing essential spectra of general weighted composition operators on $C(K)$ and other Banach lattices remains, to the best of our knowledge, unsolved. The current paper represents another step toward the solution of this problem: namely, we describe the lower semi-Fredholm and the Fredholm spectrum of a weighted composition operator on $C(K)$ in the case when the corresponding map is a surjection.

The paper is organized as follows.

\noindent In section 2 we introduce the notations we use throughout the paper and state some known results needed in the sequel.

\noindent In section 3 we obtain a criterion for the operator $\lambda I - T$ to be Fredholm or lower semi-Fredholm, providing that $|\lambda| < \rho_{min}(T)$. We refer the reader to formula~(\ref{eq12}) below for the definition of $\rho_{min}(T)$.

\noindent In section 4 we obtain a criterion for the operator $\lambda I - T$ to be Fredholm or lower semi-Fredholm, providing that $\rho_{min}(T) < |\lambda| < \rho(T)$. To obtain the corresponding results, and at the same time to avoid making our statements too cumbersome, we had to impose  the following additional conditions.
\begin{enumerate}
  \item The surjection $\varphi$ is open.
  \item The weight $w$ is an invertible element of the algebra $C(K)$.
  \item The set of all eventually $\varphi$-periodic points is of first category in $K$.
\end{enumerate}

It might be worth noticing that the criterions we obtained involve the notion of \textbf{almost homeomorphisms} of compact Hausdorff spaces that was introduced and studied by Louis Friedler and the first named author in~\cite{FK}. In particular, Theorem~\ref{t8} states that if the compact Hausdorff space belongs to the class AH (see Definition~\ref{d5}) then the spectrum $\sigma(T)$ coincides with the Fredholm spectrum $\sigma_f(T)$.

\section{Notations and preliminaries}

Throughout the paper the following notations are used.

\noindent $K$ is a compact Hausdorff space.

\noindent $C(K)$ is the Banach algebra of all continuous complex-valued functions on $K$ endowed with the supremum norm.

\noindent $\varphi$ is a continuous map of $K$ onto itself.

\noindent $w \in C(K)$ is the weight.

\noindent $\mathds{N}$ is the semigroup of all positive integers.

\noindent $\mathds{Z}$ is the group of all integers.

\noindent $\mathds{R}$ is the field of all real numbers.

\noindent $\mathds{C}$ is the field of all complex numbers.

\noindent $\mathds{T}$ is the unit circle.

\noindent $\mathds{D}$ is the open unit disk.

\noindent $\mathds{U}$ is the closed unit disk.

\noindent For any $n \in \mathds{N}$ we denote by $\varphi^n$ the $n^{th}$ iteration of the map $\varphi$.

\noindent $\varphi^0$ is the identical map: $\varphi^0(k) = k, k \in K$.

\noindent If the map $\varphi$ is invertible and $n \in \mathds{N}$ then $\varphi^{-n}$ is the $n^{th}$ iteration of the inverse map $\varphi^{-1}$.

\noindent For any subset $F$ of $K$ and for any $n \in \mathds{N}$ we denote by $\varphi^{(-n)}(F)$ the full $n^{th}$ preimage of $F$, i.e.
\begin{equation*}
\varphi^{(-n)}(F) = \{k \in K : \; \varphi^n(k) \in F\}
\end{equation*}

\noindent For any $n \geq 1$ we denote by $w_n$ the function $w(w \circ \varphi) \ldots (w \circ \varphi^{n-1})$. Thus, $w_1 = w$.

A point $k \in K$ is called eventually $\varphi$-periodic if for some $n \geq 0$ the point $\varphi^n(k)$ is $\varphi$-periodic. When it does not cause any ambiguity we will write eventually periodic or periodic instead of eventually $\varphi$-periodic or $\varphi$-periodic, respectively.

\noindent For a given $w \in C(K)$ and a continuous surjection $\varphi$ we define the weighted composition operator
$T = wT_\varphi$ as
\begin{equation*}
  (Tf)(k) = w(k)f(\varphi(k)), \; f \in C(K), \; k \in K.
\end{equation*}
Let $X$ be a Banach space. We denote its Banach conjugate by $X^\prime$.

\noindent We denote the algebra of all bounded linear operators on $X$ by $L(X)$.

\noindent Let $T \in L(X)$. We denote by $\sigma(T,X)$ and $\rho(T,X)$  the spectrum and the spectral radius of $T$, respectively.

\noindent We will use the notations $\sigma(T)$ and $\rho(T)$, as far as it does not cause any ambiguity.

\noindent In the case when $T = wT_\varphi$ is a weighted composition operator on $C(K)$ the spectral radius $\rho(T)$ can be computed as (see~\cite{Ki1})
\begin{equation}\label{eq11}
  \rho(T) = \max \limits_{\mu \in M_\varphi} \exp \int \ln |w| d\mu,
\end{equation}
where $M_\varphi$ is the set of all regular Borel $\varphi$-invariant probability measures on $K$.

\noindent In the case when $|w| >0$ on $K$ we also introduce
\begin{equation}\label{eq12}
  \rho_{min}(T) = \min \limits_{\mu \in M_\varphi} \exp \int \ln |w| d\mu.
\end{equation}

\noindent Notice that in the case when the operator $T = wT_\varphi$ is invertible on $C(K)$ we have
$\rho_{min}(T) = [\rho(T^{-1})]^{-1}$.

\noindent Recall that $T \in L(X)$ is called lower semi-Fredholm if its defect is finite, i.e. $codim(TX) < \infty$, it is called upper semi-Fredholm if its range is closed in $X$ and $dim \ker{T} < \infty$, and it is called Fredholm if it is both lower semi-Fredholm and upper semi-Fredholm

\noindent We use the standard notations $\Phi^-(X)$, $\Phi^+(X)$, and $\Phi(X)$  for the sets of all lower semi-Fredholm operators, upper semi=Fredholm operators, and Fredholm operators in $L(X)$, respectively. We will write $\Phi^-, \Phi^+$, and $\Phi$ instead of $\Phi^-(X), \Phi^+(X)$, and $\Phi(X)$ when it will not cause any ambiguity.

\noindent We consider the following subsets of $\sigma(T)$.

\noindent $\sigma_{a.p.}(T) = \{\lambda \in \mathds{C}: \exists x_n \in X, \|x_n\| =1, Tx_n -\lambda x_n \rightarrow 0$\}. Thus, $\sigma_{a.p.}(T)$ is the union of the point spectrum and the approximate point spectrum of $T$.

\noindent $\sigma_r(T) = \sigma(T) \setminus \sigma_{a.p.}(T)$; i.e., $\lambda \in \sigma_r(T)$ if and only if the operator $\lambda I - T$ is not invertible but bounded from below.

\noindent $\sigma_{usf}(T)$ is the upper semi-Fredholm spectrum of an operator $T \in L(X)$. It is defined as
\begin{equation*}
  \sigma_{usf}(T) = \{\lambda \in \mathds{C} : \; \lambda I - T \not \in \Phi^+ \}.
  \end{equation*}

\noindent $\sigma_{lsf}(T)$ is the lower semi-Fredholm spectrum of an operator $T \in L(X)$. It is defined as
\begin{equation*}
  \sigma_{lsf}(T) = \{\lambda \in \mathds{C} : \; \lambda I - T \not \in \Phi^- \}.
  \end{equation*}

  \noindent $\sigma_f(T)$ is the Fredholm spectrum of an operator $T \in L(X)$. It is defined as
\begin{equation*}
  \sigma_f(T) = \{\lambda \in \mathds{C} : \; \lambda I - T \not \in \Phi \}.
  \end{equation*}

  \noindent We will need the following well known characterizations of $\sigma_{usf}(T)$ and $\sigma_{lsf}(T)$ (see e.g.~\cite{EE}).

  \begin{proposition} \label{p1} Let $X$ be a Banach space and $T \in L(X)$. Let $\lambda \in \mathds{C}$.
  The following conditions are equivalent.
  \begin{enumerate}
    \item $\lambda \in \sigma_{usf}(T)$
    \item There is a sequence $x_n$, $x_n \in X$, $n \in \mathds{N}$, such that $\|x_n\| = 1$,
    $Tx_n - \lambda x_n \rightarrow 0$, and the sequence $x_n$ is singular, i.e. it does not contain any norm convergent subsequence.
  \end{enumerate}
  \end{proposition}

  \begin{proposition} \label{p2} Let $X$ be a Banach space and $T \in L(X)$. Let $\lambda \in \mathds{C}$.
  The following conditions are equivalent.
  \begin{enumerate}
    \item $\lambda \in \sigma_{lsf}(T)$
    \item There is a sequence $x_n'$, $x_n' \in X^\prime$, $n \in \mathds{N}$, such that $\|x_n'\| = 1$,
    $T'x_n' - \lambda x_n' \rightarrow 0$, and the sequence $x_n'$ is singular, i.e. it does not contain any norm convergent subsequence.
  \end{enumerate}
  \end{proposition}

  \begin{definition}
    A subset $\{k_n, n \in \mathds{Z}\}$ of $K$ is called a $\varphi$-string if $\varphi(k_n) = k_{n+1}, n \in \mathds{Z}$.
  \end{definition}
  Consider the space $\prod \limits_{n= - \infty}^\infty K_n$ endowed with the Tikhonov topology, where each $K_n$ is a copy of $K$. The set of all $\varphi$-strings is a closed subset of $\prod \limits_{n= - \infty}^\infty K_n$. We will denote the compact space of all $\varphi$-strings by $\mathbf{K}$.

  \noindent We define the homeomorphism $\Phi$ of $\mathbf{K}$ onto itself as follows. If $\mathbf{k} \in \mathbf{K}$ and $\mathbf{k} = (k_n)_{n=-\infty}^\infty$ then $\Phi(\mathbf{k}) = (k_{n+1})_{n=-\infty}^\infty$.

  \noindent We define $\mathbf{w} \in C(\mathbf{K})$ by the equality
  \begin{equation*}
    \mathbf{w}(\mathbf{k}) =w(k_0).
  \end{equation*}

  \begin{definition} \label{d4} Let $K$ be a compact Hausdorff space, $\varphi : K \rightarrow K$ be a continuous surjection, and $w \in C(K)$. Let $\mathbf{K}$, $\Phi$, and $\mathbf{w}$ be the objects introduced above. We will say that $\Phi$ is the homeomorphism \textbf{associated} with $\varphi$ and that $\mathbf{T} = \mathbf{w}T_\Phi$ is the operator \textbf{associated} with $T = wT_\varphi$.

  \end{definition}

  We will need the following two facts about the lower semi-Fredholm spectrum of weighted automorphisms of $C(K)$ (see~\cite{Ki3}).

  \begin{theorem} \label{t2} Let $\varphi$ be a homeomorphism of $K$ onto itself, $w \in C(K)$, and $T =wT_\varphi$. Assume that the set of all $\varphi$-periodic points is of first category in $K$. Let $\lambda \in \mathds{C} \setminus \{0\}$. The following conditions are equivalent
  \begin{enumerate}
    \item $def(\lambda I - T) \neq 0$.
    \item $\exists k \in K$ such that
    \begin{equation}\label{eq7}
      |w_n(k)| \leq |\lambda|^n \; \text{and} \; |w_n(\varphi^{-n}(k)| \geq |\lambda|^n, \; n \in \mathds{N}.
    \end{equation}
  \end{enumerate}
     \end{theorem}

 \begin{theorem} \label{t3} Let $\varphi$ be a homeomorphism of $K$ onto itself, $w \in C(K)$, and $T =wT_\varphi$. Assume that the set of all $\varphi$-periodic points is of first category in $K$. Let $\lambda \in \sigma(T) \setminus \{0\}$. The following conditions are equivalent
\begin{enumerate}
  \item $(\lambda I - T)C(K) = C(K)$.
  \item  $K = K_1 \cup K_2 \cup O$, where $K_1$, $K_2$, and $O$ are pairwise disjoint nonempty subsets of $K$ such that
      \begin{enumerate}[(a)]
        \item $K_1$ and $K_2$ are closed $\varphi$-invariant subsets of $K$.
        \item $\rho(T, C(K_1)) < |\lambda|$.
        \item The operator $T$ is invertible on $C(K_2)$ and $\rho(T^{-1}, C(K_2)) < |\lambda|^{-1}$.
        \item For any closed subset $F$ of $O$ we have
        \begin{equation}\label{eq8}
          \bigcap \limits_{n=1}^\infty cl \bigcup \limits_{j=n}^\infty \varphi^j(F) \subseteq K_2
        \end{equation}
        and
        \begin{equation}\label{eq9}
          \bigcap \limits_{n=1}^\infty cl \bigcup \limits_{j=n}^\infty \varphi^{-j}(F) \subseteq K_1.
        \end{equation}
      \end{enumerate}
\end{enumerate}
\end{theorem}

We will end this section with the following definition.

\begin{definition} \label{d3} Let $K$ be a compact Hausdorff space and $\varphi : K \rightarrow K$ be a surjection. We will say that the point $k$ is \textbf{perfectly periodic} if $k$ is periodic and $\varphi^{(-1)}(\{k, \varphi(k), \ldots \varphi^{p-1}(k)\}) = \{k, \varphi(k), \ldots \varphi^{p-1}(k)\}$, where $p$ is the period of $k$.
  \end{definition}

\section{Conditions for $\lambda I - T$ to be lower semi-Fredholm. The case $\lambda < \rho_{min}(T)$.}

In this section we will obtain a criterion for the operator $T = wT_\varphi$ to be lower semi-Fredholm, providing that $\varphi : K \rightarrow K$ is a non-invertible surjection. The answer involves the notion of \textbf{almost homeomorphism} of a compact Hausdorff space introduced in~\cite{FK}

\begin{definition} \label{d1} Let $K$ be a compact Hausdorff space and $\varphi : K \rightarrow K$ be a continuous surjection. The map $\varphi$ is called an almost homeomorphism of $K$ if there is a finite subset $S$ of $K$ such that the map $\varphi : (K \setminus S) \rightarrow (K \setminus \varphi(S))$ is a homeomorphism.
\end{definition}

\begin{definition} \label{d2} We say that a compact Hausdorff space $K$ belongs to the class $\mathcal{AH}$ if every almost homeomorphism of $K$ is a homeomorphism.
\end{definition}

\begin{theorem} \label{t4} Let $\varphi$ be a continuous non-invertible \footnote{Because the case when $\varphi$ is a homeomorphism of $K$ onto itself was considered in~\cite{Ki3}, in all the statements in the current paper we assume that the surjection $\varphi$ is not invertible.} map of a compact Hausdorff space $K$ onto itself and let $w \in C(K)$. Assume that $|w| > 0$ on $K$. Let $T = wT_\varphi$ be the corresponding weighted composition operator on $C(K)$. The following conditions are equivalent
\begin{enumerate}[(1)]
  \item The operator $\lambda I - T$ is lower semi-Fredholm for every $\lambda \in \rho_{min}(T)\mathds{D}$.
  \item The operator $\lambda I - T$ is Fredholm for every $\lambda \in \rho_{min}(T)\mathds{D}$.
  \item The map $\varphi$ is an almost homeomorphism of $K$.
\end{enumerate}
Moreover, for every $\lambda \in \rho_{min}(T)\mathds{D}$
\begin{multline}\label{eq10}
  ind(\lambda I - T) = def(\lambda I - T) = \\
   = card \{(p,q) : \; p,q \in K, p \neq q, \varphi(p) = \varphi(q) \}.
\end{multline}
\end{theorem}

\begin{proof} $(2) \Rightarrow (1)$. This implication  is trivial.

\noindent $(1) \Rightarrow (2)$. It is sufficient to prove that $\ker{(\lambda I - T)} = 0$ if $|\lambda| < \rho_{min}(T)$. Moreover, we will prove that $\rho_{min}(T)\mathds{D} \subseteq \sigma_r(T)$. Indeed, if $\lambda \in \sigma_{a.p.}(T)$ then (see~\cite{Ki1}) there is $k \in K$ such that for any $u \in \varphi^{(-n)}(k)$ and for any $n \in \mathds{N}$ we have
\begin{equation}\label{eq13}
  |w_n(u)| \leq |\lambda|^n.
\end{equation}
 It is immediate to see from~(\ref{eq13}) and~(\ref{eq12}) that $|\lambda| \geq \rho_{min}(T)$.

 \noindent $(2) \Rightarrow (3)$. This implication is also trivial.

 \noindent $(3) \Rightarrow (2)$. First notice that $T_\varphi C(K)$ is a closed unital $C^\star$ subalgebra of $C(K)$ and therefore $(3)$ implies that the operator $T_\varphi$ is Fredholm and that
 \begin{equation*}
   ind(T_\varphi) = card \{(p,q) : \; p,q \in K, p \neq q, \varphi(p) = \varphi(q) \}.
 \end{equation*}
 Because the operator of multiplication by $w$ is invertible in $C(K)$ we see that $T$ is Fredholm and $ind(T) = ind(T_\varphi)$.

 To prove the last statement of the theorem recall that the index of a semi-Fredholm operator is stable under small norm perturbations. Therefore the set $A = \{\lambda \in \rho_{min}(T)\mathds{D} : ind(\lambda I - T) = ind(T)\}$ is open in $\mathds{C}$. Assume, contrary to our claim, that there is $\alpha \in \rho_{min}(T)\mathds{D} \cap \partial A$. The operator $\alpha I - T$ is upper semi-Fredholm (because $\alpha \in \sigma_r(T)$) and $ind(\alpha I - T) \neq ind(T)$. But it contradicts the stability of index under small norm perturbations.
\end{proof}

If we do not assume that the weight $w$ is invertible in $C(K)$ the corresponding statement becomes more complicated. Let us denote by $Z(w)$ the set of all zeros of $w$ in $K$, by $P$ the subset of $Z(w)$ that consists of $\varphi$-periodic points, and by $S$ the smallest $\varphi^{(-1)}$-invariant subset of $K$ that contains $P$. Let $K_1 = K \setminus S$.

\begin{theorem} \label{t5} Let $\varphi$ be a continuous non-invertible map of a compact Hausdorff space $K$ onto itself and let $w \in C(K)$. Let $T = wT_\varphi$ be the corresponding weighted composition operator on $C(K)$. The following conditions are equivalent
\begin{enumerate}
  \item $T$ is lower semi-Fredholm.
  \item $T$ is Fredholm.
  \item $\varphi$ is an almost homeomorphism of $K$ and every point of the set $Z(w)$ is isolated in $K$.
\end{enumerate}
Moreover, if $T$ is Fredholm then $\lambda I - T$ is Fredholm for any $\lambda \in \rho_{min}(T, C(K_1))\mathds{D}$ \footnote{It follows from (3) that $K_1$ is a compact subspace of $K$.} and
\begin{multline}\label{eq14}
  ind(\lambda I - T) = def(\lambda I - T) = \\
  = card \{(p,q) : \; p,q \in K, p \neq q, \varphi(p) = \varphi(q) \}.
\end{multline}
  \end{theorem}

  \begin{proof} $(3) \Rightarrow (2)$. Let $f_n \in C(K)$ be such that $\|f_n\| = 1$ and $Tf_n \rightarrow 0$. Condition (3) guarantees that $|w| > c >0$ on $K \setminus Z(w)$. Therefore $f_n \rightarrow 0$ uniformly on the set $\varphi(K \setminus Z(w))$. Indeed, $|f_n(\varphi(k))| \leq c^{-1}\|Tf_n\|, k \in K \setminus Z(w)$. It follows from the equality $\varphi(K \setminus Z(w)) \cup \varphi(Z(w)) = K$ that $K \setminus \varphi(Z(w)) \subseteq \varphi(K \setminus Z(w))$.

  Let $B$ be the subset of $\varphi(z(w))$ that consists of points isolated in $K$. Notice that the set $B$ is either finite or empty. Let $A = \varphi(z(w) \setminus B$. We claim that $|f_n(t)| \leq c^{-1}\|Tf_n\|, t \in A$. Indeed, let us fix $t \in A$. Because the point $t$ is not isolated in $K$ and the set $\varphi(z(w))$ is finite there is a net $t_\alpha$ in $K$ that converges to $t$ and such that $t_\alpha \not \in \varphi(z(w))$.  Thus, $|f_n(t_\alpha)| \leq c^{-1}\|Tf_n\|$ and therefore, $|f_n(t)| \leq c^{-1}\|Tf_n\|, t \in A$.
  The sequence $f_n$ converges uniformly to zero on the set $K \setminus B$. Because the set $B$ is either empty or finite and consists of points isolated in $K$, there is a subsequence of the sequence $f_n$ that converges in $C(K)$.
   By Proposition~\ref{p1} $T$ is upper semi-Fredholm.

  Assume now that there is a sequence $\mu_n$ in $C(K)'$ such that $\|\mu_n\| = 1$ and $\mu_n \rightarrow 0$. Let $\nu_n = \mu_n | Z(w)$ and $\tau_n = \mu_n - \nu_n$. Then $T'\nu_n = 0$, and therefore $T'\tau_n \rightarrow o$. Let
  \begin{equation*}
 w_1(k)=   \left\{
      \begin{array}{ll}
        w(k), & \hbox{if $k \in K \setminus Z(w)$;} \\
        1, & \hbox{if $k \in Z(w)$,}
      \end{array}
    \right.
  \end{equation*}
and let $T_1 = w_1 T_\varphi$. Then $T_1'\tau_n = T'\tau_n \rightarrow 0$ and by Theorem~\ref{t4} the sequence $\tau_n$ contains a convergent subsequence. Because the sequence $\nu_n$ obviously has a convergent subsequence it follows that $\mu_n$ contains a convergent subsequence. By Proposition~\ref{p2} $T$ is lower semi-Fredholm and therefore, Fredholm.

 \noindent $(2) \Rightarrow (1)$. This implication is trivial.

\noindent $(1) \rightarrow (3)$. Assume $(1)$. If $k \in Z(w)$ and $k$ is not isolated in $K$ we can find pairwise distinct points $k_n \in K$ such that $|w(k_n)| \leq \frac{1}{n}, n \in \mathds{N}$. The sequence $\delta_{k_n} \in C(K)'$ is singular and $T'\delta_{k_n} \rightarrow 0$, a contradiction. Let us fix an $\varepsilon >0$ and let
 \begin{equation} \label{eq15}
 w_\varepsilon(k)=   \left\{
      \begin{array}{ll}
        w(k), & \hbox{if $k \in K \setminus Z(w)$;} \\
        \varepsilon, & \hbox{if $k \in Z(w)$,}
      \end{array}
    \right.
\end{equation}
For any sufficiently small $\varepsilon$ the operator $T_\varepsilon = w_\varepsilon T_\varphi$ is lower semi-Fredholm, and by Theorem~\ref{t4} $\varphi$ is an almost homeomorphism of $K$.

Having proved the equivalence of (1), (2), and (3) we will prove now the last statement of the theorem.

  Assume one of equivalent conditions (1) - (3). Let $S = \bigcup \limits_{n=1}^\infty \varphi^{(-n)}(P)$. Because $\varphi$ is an almost homeomorphism and the points of $P$ are isolated in $K$, the set $S$ is an at most countable open subset of $K$. Hence, the set $K_1 = K \setminus S$ is a compact subset of $K$ and $\varphi (K_1) = \varphi^{(-1)}(K_1) = K_1$.

  Let $\varepsilon >0$ and let $T_\varepsilon$ be as above. By Theorem~\ref{t4} for every $\lambda \in \rho_{min}(T_\varepsilon)\mathds{D}$ the operator $\lambda I - T_\varepsilon$ is Fredholm. But $\rho_{min}(T,C(K_1)) \leq \rho_{min}(T_\varepsilon)$ and the operator $T - T_\varepsilon$ is finite dimensional. Hence the last statement of the theorem follows.
   \end{proof}
   \bigskip

   \section{ Conditions for $\lambda I - T$ to be lower semi-Fredholm. The case $|\lambda| > \rho_{mim}(T)$.}

   We start with the following lemma.

   \begin{lemma} \label{l1} Let $K$ be a compact Hausdorff space and $\varphi: K \rightarrow K$ be a continuous non-invertible surjection. Let $\Phi$ be the homeomorphism of $\mathbf{K}$ associated with $\varphi$ (see Definition~\ref{d4}). The following conditions are equivalent
   \begin{enumerate}
     \item The set of perfectly $\varphi$-periodic points (see Definition~\ref{d3}) is of first category in K.
     \item The set of $\Phi$-periodic points is of first category in $\mathbf{K}$.
   \end{enumerate}
      \end{lemma}

\begin{proof} Denote by $P_n$ (respectively, $\mathbf{P}_n$) the set of all $\varphi$-periodic (respectively, $\Phi$-periodic) points of period $n$.

\noindent $(1) \Rightarrow (2)$. Assume (1) and assume to the contrary that for some $n \in \mathds{N}$ we have $Int \; \mathbf{P}_n \neq \emptyset$. Considering, if necessary, the map $\Phi^n$ instead of $\Phi$ we can assume without loss of generality that $n =1$. The set
     $E = \{\mathbf{k}_0 : \mathbf{k} \in Int \; \mathbf{P}_1\}$ is an open subset of $P_1 = \{s \in K : \varphi(s) = s\}$ and $\varphi(E) = E$. Therefore, by (1) there is an $s_1 \in K$ such that $s_1$ is not a $\varphi$-periodic point and $\varphi(s_1) = k \in E$.  Let $\mathbf{k} \in \mathbf{K}$ be such that $(\mathbf{k}_n) = k, n \in \mathds{Z}$. Let $s_n : n \geq 2$ be a sequence of points in $K$ such that
     $\varphi(s_{n+1}) = s_n, n \in \mathds{N}$.  For any $m \in \mathds{N}$ consider the point $\mathbf{k}^m \in \mathbf{K}$ such that
     \begin{equation*}
     \mathbf{k}^m_n =  \left\{
         \begin{array}{ll}
           k, & \hbox{if $n \geq -m$}; \\
          s_{-(n+m)}, & \hbox{if $n < -m$.}
         \end{array}
       \right.
     \end{equation*}
     The points $\mathbf{k}^m, m \in \mathds{N}$, are not $\Phi$-periodic and they converge to $\mathbf{k}$ in $\mathbf{K}$, a contradiction.

     \noindent $(2) \Rightarrow (1)$. Assume (2) and assume to the contrary that for some $n \in \mathds{N}$ there is an open in $K$ nonempty subset $E$ of $P_n$ such that $E = \varphi^{(-1)}(E)$. Let $\mathbf{E} = \{\mathbf{k} \in \mathbf{K} : \mathbf{k}_0 \in E\}$. Then $\mathbf{E}$ is an open nonempty subset of $\mathbf{K}$ and $\mathbf{E} \subseteq \mathbf{P}_n$, a contradiction.
\end{proof}

     \begin{corollary} \label{c2} Let $K$ be a compact Hausdorff space and $\varphi: K \rightarrow K$ be a continuous non-invertible surjection. Let $\Phi$ be the homeomorphism of $\mathbf{K}$ associated with $\varphi$ (see Definition~\ref{d4}).

     If the set of all eventually $\varphi$-periodic points is of first category in $K$, then
the set of $\Phi$-periodic points is of first category in $\mathbf{K}$.
     \end{corollary}

We will proceed with proving a series of lemmas needed for our main result in this section, Theorem~\ref{t6}.

\begin{lemma} \label{l2}. Let $K$ be a compact Hausdorff space, $\varphi : K \rightarrow K$ be a continuous non-invertible surjection, $w \in C(K)$, and $T = wT_\varphi$. Assume that
\begin{enumerate}
  \item $\lambda \in \sigma(T) \setminus \sigma(\mathbf{T})$ and $|\lambda| > \rho_{min}(T)$.
 \item The operator $\lambda I - T$ is lower semi-Fredholm.
\item The set of eventually $\varphi$-periodic points is of first category in $K$.
 \noindent Then $K = K_1 \cup K_2$ where $K_1$ and $K_2$ are nonempty closed disjoint subsets of $K$ such that
  \end{enumerate}
  \begin{itemize}
    \item $\varphi(K_i) = K_i, i = 1,2$.
    \item $K_2 \neq \emptyset$ and $\rho_{min}(T, C(K_2)) > |\lambda|$.
    \item The restriction of $\varphi$ on $K_2$ is an almost homeomorphism, but not a homeomorphism of $K_2$.
     \item $\rho(T, C(K_1)) < |\lambda|$ .
  \end{itemize}
\end{lemma}

\begin{proof} By (3), Lemma~\ref{l2}, and Theorem 3.7 in~\cite{Ki1} the spectrum $\sigma(\mathbf{T})$ is rotation invariant and therefore $\lambda \mathds{T} \cap \sigma(\mathbf{T}) = \emptyset$. By Theorem 3.10 in~\cite{Ki1} we have $\mathbf{K} = \mathbf{K}_1 \cup \mathbf{K}_2$ where $\mathbf{K}_1$ and $\mathbf{K}_2$ are disjoint $\Phi$-invariant closed subsets of $\mathbf{K}$, $\rho(\mathbf{T}, C(\mathbf{K}_1)) < |\lambda|$, the operator $\mathbf{T}$ is invertible on $C(\mathbf{K}_2)$ and $\rho(\mathbf{T}^{-1}, C(\mathbf{K}_2)) < |\lambda|$. Notice that the sets $\mathbf{K}_1$ and $\mathbf{K}_2$ cannot be empty because on the one hand, $\rho_{min}(\mathbf{T}) = \rho_{min}(T) > |\lambda|$, while on the other hand, $\rho(\mathbf{T}) = \rho(T) > |\lambda|$. Let $p : \mathbf{K} \rightarrow K$ be the map defined as $p(\mathbf{k}) = \mathbf{k}_0$ and let $K_1 = p(\mathbf{K_1})$ and $K_2 = p(\mathbf{K}_2)$. It is immediate to see that $K_1$ and $K_2$ are disjoint $\varphi$-invariant closed subsets of $K$ and that $\rho(T, C(K_1)) < |\lambda|$ while $\rho_{min}(T, C(K_2)) > |\lambda|$. The map $\varphi$ cannot be a homeomorphism of $K_2$ onto itself because by theorem 3.10 in~\cite{Ki1} it would imply that $\lambda \not \in \sigma(T)$, and by Theorem~\ref{t4} it must be an almost homeomorphism of $K_2$.
\end{proof}

\begin{lemma} \label{l3} Let $K$ be a compact Hausdorff space, $\varphi : K \rightarrow K$ be a continuous non-invertible surjection, $w \in C(K)$, and $T = wT_\varphi$. Assume that

\begin{enumerate}
\item $w$ is invertible in $C(K)$.
\item The map $\varphi$ is open.
  \item $\lambda \in \sigma(T) \cap \sigma(\mathbf{T})$ and $|\lambda| > \rho_{min}(T)$.
  \item The operator $\lambda I - T$ is lower semi-Fredholm.
  \item $(\lambda I - \mathbf{T})C(\mathbf{K}) = C(\mathbf{K})$.
    \item The set of eventually $\varphi$-periodic points is of first category in $K$.
  \end{enumerate}
 Then there are subsets $K_1$, $K_2$, and $Q$ of $K$ such that
\begin{enumerate}[(a)]
  \item The set $K_1$ is closed in $K$ and $\varphi(K_1)=\varphi^{(-1)}(K_1) = K_1$.
  \item $\rho(T, C(K_1)) < |\lambda|$.
  \item The set $K_2$ is closed in $K$, $\varphi(K_2) = K_2$, $\rho_{min}(T,C(K_2)) > |\lambda|$ and $\varphi$ is an almost homeomorphism of $K_2$.
      \item If
      \begin{equation*}
    K = K_1 \cup \bigcup \limits_{n=0}^\infty \varphi^{(-n)}(K_2)
  \end{equation*}
       then the set $K_2$ is open in $K$.
    \item The set $Q$,
    \begin{equation*}
   Q = K \setminus \Big{(} K_1 \cup \bigcup \limits_{n=0}^\infty \varphi^{(-n)}(K_2) \Big{)}
  \end{equation*}
   is open in $K$ and $\varphi(Q) = \varphi^{(-1)}(Q) = Q$.
  \item The sets $K_1, K_2$, and $Q$ are pairwise disjoint and
  \begin{equation*}
    K = K_1 \cup \bigcup \limits_{n=0}^\infty \varphi^{(-n)}(K_2) \cup Q.
  \end{equation*}

  \item Assume that $Q \neq \emptyset$ and let $E$ be a closed in $K$ subset of $Q$ and $V_1, V_2$ be open neighborhoods of $K_1$ and $K_2$, respectively. Then there is an $m \in \mathds{N}$ such that for any $n \geq m$ we have $\varphi^n(E) \subset V_2$ and $\varphi^{(-n)}(E) \subset V_1$.
   \item Assume that $Q \neq \emptyset$. Then there is an open neighborhood $V$ of $K_2$ such that $\varphi(V \cap Q) \subset V \cap Q$ and the restriction of $\varphi$ on $V \cap Q$ is one-to-one.
\end{enumerate}
\end{lemma}

\begin{proof} Conditions (3) and (5) combined with Lemma~\ref{l1} and Theorem~\ref{t3} provide that $\mathbf{K}$ is the union of three disjoint nonempty $\Phi$ and $\Phi^{-1}$-invariant subsets $\mathbf{K}_1, \mathbf{K}_2$, and $\mathbf{O}$ with the properties

\begin{enumerate}[(I)]
  \item The sets $\mathbf{K}_1$ and $\mathbf{K}_2$ are closed in $\mathbf{K}$.
  \item $\mathbf{T}$ is invertible on $C(\mathbf{K}_2)$ and $\rho(\mathbf{T}^{-1}, C(\mathbf{K}_2)) < 1$.
  \item $\rho(\mathbf{T}, C(\mathbf{K}_1)) < 1$.
  \item If $\mathbf{V}_1$ and $\mathbf{V}_2$ are open neighborhoods of $\mathbf{K}_1$ and $\mathbf{K}_2$, respectively, and $\mathbf{E}$ is a closed subset of $\mathbf{O}$ than there is an $m \in \mathds{N}$ such that for any $n \geq m$ we have
      \begin{equation*}
        \Phi^n(\mathbf{E}) \subseteq \mathbf{V}_2 \; \text{and} \; \Phi^{-n}(\mathbf{E}) \subseteq \mathbf{V}_1.
      \end{equation*}
\end{enumerate} Let $p: \mathbf{K} \rightarrow K$ be the map $p(\mathbf{k})=k_0$ introduced in the proof of Lemma~\ref{l2} and let
$K_1 = p(\mathbf{K_1})$, $K_2 = p(\mathbf{K}_2)$, and $O = p(\mathbf{O})$. We proceed with proving the statements (a) - (i) of the lemma.

(a). Obviously the set $K_1$ is closed. It follows from the equalities $\Phi(\mathbf{K}_1) = \mathbf{K}_1$, $\Phi(\mathbf{O}) = \mathbf{O}$, and $\Phi(\mathbf{K}_2) = \mathbf{K}_2$ and the definition of the map $p$ that $\varphi(K_1) = K_1$, $\varphi(O) = O$, and $\varphi(K_2) = K_2$. Moreover, it follows from properties (II) - (IV) that $K_1$ does not intersect with $K_2 \cup O$ and therefore, $\varphi(K_1) = K_1 = \varphi^{-1}(K_1)$.

(b). Follows from (III) and the definition of $p$.

(c). The inequality $\rho(\mathbf{T}^{-1}, C(\mathbf{K}_2)) < 1$ follows from (II). It follows from the fact that the operator $\lambda I - T$ is lower semi-Fredholm and Theorem~\ref{t4} that $\varphi$ is an almost homeomorphism of $K_2$.

(d). It follows from Baire Category Theorem and the fact that $K_1 \cap F = \emptyset$ that there is an $n \in \mathds{N}$ such that $Int \, \varphi^{(-n)}(K_2) \neq \emptyset$. Then, because $\varphi$ is open, $Int \, K_2 \neq \emptyset$. Assume, contrary to our claim, that $Int \, K_2 \varsubsetneqq K_2$. Let $\tilde{K} = K \setminus \bigcup \limits_{n=0}^\infty \varphi^{(-n)}(Int \, K_2)$. Notice that $\tilde{K} = K_1 \cup \bigcup \limits_{n=0}^\infty \varphi^{(-n)}(E)$ where $E = K_2 \setminus Int \, K_2$. Applying again Baire category theorem we see that there is an $n \in \mathds{N}$ such that $Int_{\tilde{K}} \varphi^{(-n)}(E) \neq \emptyset$. Therefore, there is an open subset $V$ of $K$ such that $V \cap \varphi^{(-n)} \subset \varphi^{(-n)}$. Then $\varphi^n(V) \cap E \subset E \neq \emptyset$. The set $Int \, K_2 \cup \varphi^n(V) \cap E$ is open in $K$ and $ Int \, K_2 \subsetneqq Int \, K_2 \cup (\varphi^n(V) \cap E)$, a contradiction.

(e) - (h). For brevity let $F = \bigcup \limits_{n=0}^\infty \varphi^{(-n)}(K_2)$. We assume that $Q \neq \emptyset$, otherwise the statement is trivial.
Let $k \in Q$ and  let $E = p^{-1}(\{k\})$. $E$ is a compact subset of $\mathbf{O}$ and therefore by Theorem~\ref{t3} for arbitrary open neighborhoods $\mathbf{V}_1$ and $\mathbf{V}_2$ of, respectively, $\mathbf{K}_1$ and $\mathbf{K}_2$ there is an $N \in \mathds{N}$ such that $\Phi^n(E) \subseteq \mathbf{V}_2$ and $\Phi^{-n}(E) \subseteq \mathbf{V}_1$ for any $n \geq N$. Define a function $f$ on $K$ as follows,

\begin{equation*}
f(x) =
\begin{cases} 1     , & \text{if $x = k$,}
\\
 \lambda^i/w_i(k) , & \text{if $x = \varphi^{(i)}(k), i \in \mathds{N}$,}
\\
\lambda^{-i} w_i(x) , & \text{if $x \in \varphi^{(-i)}(\{k\}), i \in \mathds{N}$,}
\\
 0, &\text{ otherwise,}
\end{cases}
\end{equation*}
It is immediate to see that $f$ is an element of $C(K)^{\prime \prime}$ and $T^{\prime \prime} f =  \lambda f$. Because we assumed that $\lambda I - T$ is a semi-Fredholm operator it follows that (see e.g.~\cite[Theorem 4.42]{AA}) $\ker(\lambda I - T) \neq \mathbf{0}$. Let $f \in C(K)$, $f \neq 0$, and $Tf = \lambda f$. Let $U = \{k \in K : \; f(k) \neq 0\}$. It follows from $Tf = \lambda f$, $\lambda \neq 0$, and $w \in C(K)^{-1}$ that
\begin{equation}\label{eq22}
  \varphi(U) = U = \varphi^{(-1)}(U) \; \text{and} \; U \cap (K_1 \cup F) = \emptyset.
\end{equation}
By Zorn's lemma there is a maximal by inclusion open subset $V$ of $K$ with properties~(\ref{eq22}). We claim that
\begin{equation}\label{eq23}
  V \cup K_1 \cup F = K.
\end{equation}
To prove it consider the compact space $G = K \setminus V$. It follows from~(\ref{eq22}) that $\varphi(G) = G$ and therefore the operator $T$ is defined on $C(G)$. Moreover, the operator $\lambda I - T$ is lower semi-Fredholm. Assume that $K_1 \cup F \varsubsetneqq G$. Then our previous reasoning shows that there is $g \in C(G)$ such that $g \neq 0$ and $Tg = \lambda g$. Let $H = \{k \in G : g(k) \neq 0\}$. Then the set $V \cup H$ is open in $K$ and satisfies~(\ref{eq22}) in contradiction with the maximality of $V$.

We claim that for any compact subset $W$ of $V$ and for any open neighborhoods $U_1$ and $U_2$ of $K_1$ and $K_2$, respectively, there is an $N \in \mathds{N}$ such that for any $n > N$ we have
\begin{equation}\label{eq24}
  \varphi^n(W) \subset U_2 \; \text{and} \; \varphi^{(-n)}(W) \subset U_1.
\end{equation}
Indeed,~(\ref{eq24}) follows from the fact that $p^{(-1)}(W)$ is a compact subset of $\mathbf{O}$.

Notice that it follows from~(\ref{eq24}) that $cl V \cap K_2 \neq \emptyset$.

Next we claim that there is an open neighborhood  $Q$ of $K_2$ in $K$ such that $\varphi(Q \cap V) \subseteq Q \cap V$ and the restriction of $\varphi$ on $Q \cap V$ is a homeomorphism of $Q \cap V$ onto its image. First notice that $\ker{(\lambda I - T')} = \ker{((\lambda I - T'), C(K_2))} $. Indeed, let $\mu \in C(K)'$ be such that $T"\mu = \mu$. Let $U$ be an open neighborhood of $K_2$ and let $f \in C(K)$ be such that $supp \, f \subseteq K \setminus U$. Then $\int f d\mu = \lambda^{-n} \int T^n fd\mu \mathop \rightarrow \limits_{n \to \infty} 0$, because $supp \, T^n f \subseteq K \setminus \varphi^{(-n)}(U)$, $\bigcap \limits_{n=1}^\infty (K \setminus \varphi^{(-n)}(U)) = K_1$, and $\rho(T, C(K_1)) < |\lambda|$. It follows that $supp \, \mu \subseteq K_2$.

Assume, contrary to our claim, that for any open neighborhood $Q$ of $K_2$ there are $p, q \in Q \cap V$ such that $\varphi(p) = \varphi(q)$. The proof of Lemma 5.9 in~\cite{Ki3} shows that there is a sequence $\mu_n \in (C(K)'$ such that $\|\mu_n\| = 1$, $T'\mu_n - \lambda \mu_n \rightarrow 0$, and for each $n$ the measure $\mu_n$ is a finite linear combination of point measures $\delta_{s_i}$ where $S_i \in V$. Because the operator $\lambda I - T'$ is lower semi-Fredholm the sequence $\mu_n$ must contain a norm convergent subsequence. Let the limit of such a subsequence be $\nu$ then $T'\nu = \nu$ and by our previous step $supp \, \nu \subseteq K_2$. That means that for every $n$ we have $|\nu| \wedge |\mu_n| =0$, a contradiction.

Thus, there is an open neighborhood $Q$ of $K_2$ such that the restriction of $\varphi$ on $Q \cap V$ is one-to-one. The proofs of Lemma 5.10 and Corollary 5.11 in~\cite{Ki3} show that we can choose $Q$ in such a way that $\varphi(V \cap Q) \subset V \cap Q$.
\end{proof}

\bigskip

We will now consider what happens when the operator $\lambda I - T$ is lower semi-Fredholm and
$\lambda \in \sigma_{a.p.}(\mathbf{T}^\prime)$. In this case by Theorem~\ref{t2} there is $\mathbf{k} \in \mathbf{K}$ such that
\begin{equation}\label{eq25}
|\mathbf{w}_n(\mathbf{k})| \leq |\lambda|^n, \; |\mathbf{w}_n(\Phi^{-n}(\mathbf{k}))| \geq |\lambda|^n, n \in \mathds{N}.
   \end{equation}

   \begin{lemma} \label{l4} Assume conditions of Lemma~\ref{l3}. Let the operator $\lambda I - T$ be lower semi-Fredholm and
   $\lambda \in \sigma_{a.p.}(\mathbf{T}^\prime)$. Let $\mathbf{S}$ be the set of $\Phi$-strings defined as
   \begin{multline*}
     \mathbf{S} =\{\mathbf{s} : \mathbf{s} =\{\Phi^n (\mathbf{k}): n \in \mathds{Z}\} \\
     |\mathbf{w}_n(\mathbf{k})| \leq |\lambda|^n, \; |\mathbf{w}_n(\Phi^{-n}(\mathbf{k}))| \geq |\lambda|^n, n \in \mathds{N} \}.
        \end{multline*}

        Then the set $\mathbf{S}$ is finite and for every $\mathbf{s} =\{\Phi^n (\mathbf{k}): n \in \mathds{Z}\} \in \mathbf{S}$ the point $\mathbf{k}$ is isolated in $\mathbf{K}$.
   \end{lemma}

   \begin{proof} Let $\mathbf{k}$ be a point in $\mathbf{K}$ such that inequalities~(\ref{eq25}) are satisfied. Let $\{k_n : n \in \mathds{Z}\}$ be the corresponding $\varphi$-string. We have to consider several cases.

   Case 1. The set $\{k_n : n \in \mathds{N}\}$ is infinite, i.e. the point $k_0$ is not eventually periodic. We claim that
   \begin{equation}\label{eq26}
     \sum \limits_{n=1}^\infty \frac{|\lambda|^n}{|w_n(k_{-n})|}  + \sum \limits_{n=0}^\infty \frac{|w_n(k_0)|}{|\lambda|^n} < \infty.
   \end{equation}
    To prove~(\ref{eq26}) assume first to the contrary that
    \begin{equation}\label{eq27}
      \sum \limits_{n=1}^\infty \frac{|\lambda|^n}{|w_n(k_{-n})|} = \infty.
    \end{equation}
    For every $m \in \mathds{N}$ consider the discrete measure $\mu_m = \sum \limits_{n = -m}^0 \frac{\lambda^n}{w_n(k_{n})} \delta_{k_n}$. Then
\begin{equation}\label{eq28}
  \|\mu_m\| = \sum \limits_{n=0}^m \frac{|\lambda|^n}{|w_n(k_{-n})|} \mathop \rightarrow \limits_{m \rightarrow \infty} \infty
\end{equation}
and
\begin{equation}\label{eq29}
  \|T'\mu_m - \lambda \mu_m\| = \|\frac{\lambda^m}{w_m(k_{-m})}\delta_{k_{-m}} + w(k_0)\delta_{k_1}\| \leq 1 + \|w\|_{C(K)}.
\end{equation}
Let $\nu_m = \mu_m/\|\mu_m\|$ It follows from~(\ref{eq28}) and~(\ref{eq29}) that we can find positive integers $m_l$ such that $\|T'\nu_{m_l} - \lambda \nu_{m_l}\| \mathop \rightarrow  \limits_{l \rightarrow \infty} 0$ and the sequence $\nu_{m_l}$ is singular, in contradiction with our assumption that the operator $\lambda I - T$ is lower semi-Fredholm.

Similarly we can prove that
\begin{equation*}
 \sum \limits_{n=0}^\infty \frac{|w_n(k_0)|}{|\lambda|^n} < \infty.
\end{equation*}
Let
\begin{equation*}
  \mu = \sum \limits_{n=1}^\infty \frac{\lambda^n}{w_n(k_{-n})}\delta_{k_{-n}} + \sum \limits_{n=0}^\infty \frac{w_n(k_0)}{\lambda^n} \delta_{k_n}.
  \end{equation*}
Then $T^\prime \mu = \lambda \mu$. We claim that the point $k_0$ is isolated in $K$. Indeed, otherwise using~(\ref{eq26}) and the fact that the map $\varphi$ is open we can find points $k_{i,m}, m \in \mathds{N}, -m \leq i \leq m$ such that
\begin{enumerate}[(a)]
  \item $\varphi(k_{i,m} = k_{i+1,m}, m \in \mathds{N}, -m \leq i \leq m-1$.
  \item The points $k_{i,m}$ are pairwise distinct.
  \item $\mu_m - \mu \mathop \rightarrow \limits_{m \rightarrow \infty} 0$, where
 \begin{equation*}
  \mu_m = \sum \limits_{i=-m}^{-1} \frac{\lambda^{|i|}}{w_i(k_{m,i})}\delta_{k_{m,i}} + \sum \limits_{i=0}^m \frac{w_i(k_{m,i})}{\lambda^i} \delta_{k_{m,i}}.
  \end{equation*}
  Let $\nu_m = \mu_m/\|\mu_m\|$. It follows from (a) - (c) that $T^\prime \nu_m - \lambda \nu_m \rightarrow 0$ and the sequence $\nu_m$ is singular in contradiction with $\lambda I - T$ assumed to be lower semi-Fredholm.
\end{enumerate}

Case 2. The set $\{k_n : n \in \mathds{Z}\}$ is infinite but the set $\{k_n : n \in \mathds{N}\}$ is finite. In other words, there is $s \in \mathds{Z}$ such that the point $k_s$ is $\varphi$-periodic, but the point $k_{s-1}$ is not. Let $p$ be the period of the point $k_s$. There are two possibilities.

(1) $w_p(k_s) = \lambda^p$. Let
  \begin{equation*}
    \mu = \sum \limits_{i=0}^{p-1} \lambda^{p-i-1}(T^\prime)^i\delta_s.
  \end{equation*}
  Then $T^\prime \mu = \lambda \mu$.

  (2) $w_p(k_s) \neq \lambda^p$. Like in case 1 we can prove that

\begin{equation*}
      \sum \limits_{n=1}^\infty \frac{|\lambda|^n}{|w_n(k_{s-n})|} < \infty.
    \end{equation*}
    Let
 \begin{equation*}
    \mu = (\lambda^p - w_p(k_s)  \sum \limits_{n=1}^\infty \frac{\lambda^n}{|w_n(k_{s-n})|} \delta_{k_{s-n}}
    + \sum \limits_{i=0}^{p-1} \lambda^{p-i-1}(T^\prime)^i\delta_s.
    \end{equation*}
    Then $T^\prime \mu = \lambda \mu$.

    In both cases we can prove that the point $k_s$ is isolated in $K$ using the same reasoning as in Case 1.

    Case 3. The set $\{k_n : n \in \mathds{Z}\}`$ is finite.  Assume that the point $k_0$ is not isolated in $K$. Then there are the following possibilities.

    1. For every open neighborhood $V$ of $k_0$ and for every $n \in \mathds{N}$ the neighborhood $V$ contains either a point that is not $\varphi$-periodic or a $\varphi$-periodic point of period at least $n$. In this case we can construct a singular sequence $\mu_n$ such that $T^\prime \mu_n - \lambda \mu_n \rightarrow 0$.

    Indeed we can find points $t_{n,i} \in K, n \in \mathds{N} -n \leq i \leq n$ such that
\begin{enumerate}[(I)]
    \item $\varphi(t_{n,i}) = t_{n,i+1}, -n \leq i < n$,
  \item The points $t_{n,i}, n \in \mathds{N}, -n \leq i \leq n$ are pairwise distinct,
  \item $|w_j(t_{n,0})| \leq 2|\lambda|^n$ and $|w_j(t_{n, -j})| \geq 1/2 |\lambda|^n, , j=1, \ldots, n$.
\end{enumerate}
Let
\begin{equation} \label{eq30}
  \mu_n = \sum \limits_{j=0}^{n-1} \big{(} 1 - \frac{1}{\sqrt{n}}\big{)}^j
  \lambda^{-j} w_j(t_{j,0}) \delta_{t_{j,n}} + \sum \limits_{j=1}^{n-1}\big{(} 1 - \frac{\lambda^j}{\sqrt{n}}\big{)}^j \frac{1}{w_j(t_{n, -j})} \delta_{t_{n, -j}}.
\end{equation}
It is not difficult to see from (III) and~(\ref{eq30}) that
\begin{equation}\label{eq31}
  \|T'\mu_n - \lambda \mu_n\| = o(\|\mu_n\|), n \rightarrow \infty.
\end{equation}
But the measures $\mu_n$ are pairwise disjoint in virtue of II and therefore the sequence $\nu_n = \frac{\mu_n}{\|\mu_n\|}$ is singular, a contradiction.

    2. There are an open neighborhood $V$ of $k_0$ and an $n \in \mathds{N}$ such that every point in $V$ is $\varphi$-periodic and has period less or equal to $n$. That obviously contradicts our assumption that the set of eventually $\varphi$-periodic points is of first category in $K$.

    Thus it follows from our assumption that the set of all eventually $\varphi$-periodic points is of first category that only the first case is possible. Moreover, because in this case to each $\mathbf{s \in \mathbf{S}}$ there is a unique up to a constant factor discrete measure $\mu$ on the set $p(\mathbf{s})$ such that $T^\prime \mu = \lambda \mu$, we see that the set $\mathbf{S}$ is finite.
   \end{proof}

   \begin{corollary} \label{c1} Assume conditions of Lemma~\ref{l3}. Let the operator $\lambda I - T$ be lower semi-Fredholm and let the set $\mathbf{S}$ be not empty. Then there is a countable open subset $S$ of $K$ such that $S$ is the union of finite number of strings, $\varphi(K \setminus S) = K \setminus S$ and the operator $\lambda I - T$ is lower semi-Fredholm on $C(K \setminus S)$.
      \end{corollary}

      \begin{proof} We take $S = p(\mathbf{S})$ and apply Lemma~\ref{l4}.
      \end{proof}

      Finally we can state our main result.

      \begin{theorem} \label{t6} Let $K$ be a compact Hausdorff space and $\varphi$ be an open continuous non-invertible map of $K$ onto itself. Let $w$ be an invertible element of the algebra $C(K)$. Assume that the set of all eventually $\varphi$-periodic points is of first category in $K.$ Let
      \begin{equation*}
        (Tf)(k) = w(k)f(\varphi(k)), f \in C(K), k \in K.
      \end{equation*}
        Let $\lambda \in \mathds{C}$ be such that $\lambda \in \sigma(T)$ and $|\lambda| > \rho_{min}(T)$.

        \noindent The following conditions are equivalent.

        \noindent (1) The operator $\lambda I - T$ is lower semi-Fredholm.

        \noindent (2) One of the following conditions is satisfied.
        \begin{enumerate}[(I)]
          \item The compact space $K$ is the union of two disjoint $\varphi$-invariant clopen subsets $K_1$ and $K_2$ such that $\rho(T,C(K_1)) < |\lambda|$, $\rho_{min}(T,C(K_2)) > |\lambda|$, and $\varphi$ is an almost homeomorphism but not a homeomorphism of $K_2$ onto itself.
          \item There are closed subsets $K_1$ and $K_2$ of $K$ such that $\varphi^{(-1)}(K_1) = \varphi(K_1)=K_1$, $\rho(T,C(K_1) < |\lambda|$, $\varphi(K_2) = K_2$, $\rho_{min}(T,C(K_2)) > |\lambda|$, and $\varphi$ is an almost homeomorphism but not a homeomorphism of $K_2$ onto itself.

              Moreover, the set $K_2$ is a clopen subset of $K$, $K_2 \subsetneqq \varphi^{(-1)}(K_2)$, and
              \begin{equation*}
                K = K_1 \cup \bigcup \limits_{n=0}^\infty \varphi^{(-n)}(K_2).
              \end{equation*}
          \item  \begin{equation*}
    K = K_1 \cup \bigcup \limits_{n=0}^\infty \varphi^{(-n)}(K_2) \cup Q,
  \end{equation*}
  where the sets $K_1$, $K_2$, and $Q$ have properties described in the statement of Lemma~\ref{l3}.
          \item  \begin{equation*}
    K = K_1 \cup \bigcup \limits_{n=0}^\infty \varphi^{(-n)}(K_2) \cup Q \cup S,
  \end{equation*}
  where the sets $K_1$, $K_2$, and $Q$ have properties described in the statement of Lemma~\ref{l3} and $S$ is an open countable subset of $K$ disjoint with the set $K_1 \cup \bigcup \limits_{n=0}^\infty \varphi^{(-n)}(K_2) \cup Q$. Moreover, the set $S$ is the union of finite number of $\varphi$-strings  and for every $s \in S$ we have
  \begin{equation*}
    |w_n(s_0)| \leq |\lambda|^n, \; |w_n(s_{-n})| \geq |\lambda|^n, n \in \mathds{N}
  \end{equation*}
          \end{enumerate}
      \end{theorem}

      \begin{proof} The implication $(1) \Rightarrow (2)$ follows from Lemmas~\ref{l2} -~\ref{l4}.

      To prove the implication $(2) \Rightarrow (1)$ we will prove that every of conditions (I) - (IV) implies (2).

      \noindent $(I) \Rightarrow (2)$. This implication follows from Theorem~\ref{t4}.

      \noindent $(II) \Rightarrow (2)$. First we notice that (II) implies that $\lambda \in \sigma_r(T)$, and therefore the image $(\lambda I - T)C(K)$ is closed in $C(K)$. Indeed, assume to the contrary that $\lambda \in \sigma_{a.p.}(T)$. Then (see~\cite{Ki1}) there is $k \in K$ such that
      \begin{equation}\label{eq32}
        |w_n(k)| \geq |\lambda|^n \; \text{and} \; \forall t \in \varphi^{(-n)}(k), |w_n(t)| \leq |\lambda|^n, n \in \mathds{N}.
      \end{equation}
      But it is immediate to see that the existence of a point $k$ satisfying~(\ref{eq32}) contradicts condition (II).

      Assume that $\mu \in C(K)^\prime$, $\mu \neq 0$, and $T^\prime \mu = \lambda \mu$. The proof of Theorem 5.14 in~\cite{Ki3} shows that $supp \; \mu \subseteq K_2$. Because $\varphi$ is an almost homeomorphism of $K_2$ we see that $def(\lambda I - T) < \infty$.

      \noindent $(III) \Rightarrow (2)$. Assume (III) and assume to the contrary that there is a singular sequence $\mu_n \in C(K)^\prime$ such that $T^\prime \mu_n - \lambda \mu_n \rightarrow 0$. From the previous step we conclude that without loss of generality we can assume that $supp \; \mu_n \subset Q$. Assume first that the compact space $K$ is extremally disconnected. Then the restriction of $\varphi$ on $cl V$ is one-to one. Then we come to contradiction as in the proof of Theorem 5.14 in~\cite{Ki3}.

      If $K$ is not extremally disconnected we can consider the operator $T^{\prime \prime}$ on the second dual $C(K)^{\prime \prime} \cong C(\mathds{Q})$ where the compact space $\mathds{Q}$ is extremally disconnected. Notice that $T^{\prime \prime} = w^{\prime \prime} T_\psi$ where $\psi$ is a continuous map of $\mathds{Q}$ onto itself. It is not difficult to see that $\psi$ satisfies condition (III).

      \noindent $(IV) \Rightarrow (2)$. Assume (IV) and assume to the contrary that there is a singular sequence $\mu_n \in C(K)^\prime$ such that $T^\prime \mu_n - \lambda \mu_n \rightarrow 0$. Without loss of generality we can assume that $supp \mu_n \subseteq S$ Let $S^{\star \star}$ be the set $j^{-1}(S)$ where $j: Q \rightarrow K$ is the surjection corresponding to the isometric embedding of $C(K)$ into C(Q). The set $S^{\star \star}$ is a finite union of $\psi$-strings and points of $S^{\star \star}$ are isolated in $Q$. The map $\psi$ extends to a homeomorphism of $cl S^{\star \star}$ onto itself. It follows from Theorem 2.11 in~\cite{Ki3} that the sequence $\mu_n$ contains a convergent subsequence, a contradiction.
      \end{proof}

       \begin{corollary} \label{c7} Assuming one of conditions (I) - (IV) from Theorem~\ref{t6} is satisfied, the defect of $\lambda I - T$ can be computed as
       \begin{equation}\label{eq33}
         def(\lambda I - T) =card( \{(p,q) : p,q \in K_2,p \neq q, \varphi (p)= \varphi(q)\}) + card( \mathbf{S}),
       \end{equation}
where $\mathbf{S}$ is the set introduced in the statement of Lemma~\ref{l4}.

\noindent In particular, $def(\lambda I - T) =0$, i.e. $(\lambda I - T)C(K)=C(K)$ if the following two conditions are satisfied
\begin{itemize}
  \item The set $\mathbf{S}$ is empty,
  \item The map $\varphi$ is a homeomorphism of $K_2$ onto itself. \footnote{See also Theorem 5.14 in~\cite{Ki3}}
\end{itemize}
      \end{corollary}

      From Theorem~\ref{t6} easily follows the following criterion for the operator $\lambda I - T$ (where $\lambda > \rho_{min}(T)$) to be Fredholm.

      \begin{theorem} \label{t7} Let $K$ be a compact Hausdorff space, $\varphi$ be a continuous open non-invertible map of $K$ onto itself. Let $w \in C(K)^{-1}$ and
        \begin{equation*}
          (Tf)(k)=w(k)f(\varphi(k)), \; f \in C(K), k \in K.
        \end{equation*}
        Assume that the set of all eventually $\varphi$-periodic points is of first category in $K$.
        Let $\lambda \in \sigma(T)$ and $\lambda > \rho_{min}(T)$. The following conditions are equivalent.

        \noindent (1) The operator $\lambda I - T$ is Fredholm.

        \noindent (2) One of conditions (I) - (IV) is satisfied. Moreover, if the set $Q$ is not empty, then every point of $Q$ is isolated in $K$ and there is a finite subset $\{k_1, \ldots k_p\}$ of $Q$ such that the sets $A_1, \ldots , A_p$ are pairwise disjoint and $Q = \bigcup \limits_{j=1}^p A_j$, where $A_j$ is the smallest $\varphi$ and $\varphi^{(-1)}$ invariant subset of $Q$ that contains $k_j$, i.e.
        $A_j = \bigcup \limits_{n=0}^\infty \bigcup \limits_{m=1}^\infty \varphi^{(-m)}(\varphi^n(k_j))$.
      \end{theorem}

      \begin{corollary} \label{c9} Assume condition (2) of Theorem~\ref{t7}. Then
      \begin{equation}\label{eq34}
        ind(\lambda I - T) = p - def(\lambda I - T),
      \end{equation}
        where $p$ is from the statement of Theorem~\ref{t7} and $def(\lambda I - T)$ is defined by~(\ref{eq33}).
      \end{corollary}

      \begin{corollary} \label{c4} Assume either conditions of Theorem~\ref{t6}, or of Theorem~\ref{t7}. Then the spectrum $\sigma_{sfl}(T)$ or, respectively, the Fredholm spectrum $\sigma_f(T)$ is rotation invariant

      \end{corollary}

      \begin{corollary} \label{c3} Let $K$ be a compact Hausdorff space, $\varphi$ be a continuous open map of $K$ onto itself. Let $w \in C(K)^{-1}$ and
        \begin{equation*}
          (Tf)(k)=w(k)f(\varphi(k)), \; f \in C(K), k \in K.
        \end{equation*}
         Assume that the set of all eventually $\varphi$-periodic points is of first category in $K$.
        Assume that either $|\lambda|=\rho(T)$ or $|\lambda|=\rho_{min}(T)$. Then the operator $\lambda I - T$ cannot be lower semi-Fredholm.
      \end{corollary}

      \begin{proof} Let $|\lambda|=\rho(T)$. If the operator $\lambda I - T$ is semi-Fredholm then by the punctured neighborhood theorem (see~\cite{Sch}) $\lambda$ is an isolated point in $\sigma(T)$. On the other hand, it follows from our assumption that the set of all eventually $\varphi$-periodic points is of first category in $K$ that $\sigma(T)$ is rotation invariant (see~\cite{Ki1}). As $\rho(T) >0$, we have a contradiction.

      Next, assume that $|\lambda| = \rho_{min}(T)$ and the operator $\lambda I - T$ is lower semi-Fredholm. Then there are $\lambda_n \in \mathds{C}, n \in \mathds{N}$ such that $|\lambda_n| \downarrow |\lambda|$ and either $\lambda_n \not \in \sigma(T), n \in \mathds{N}$, or $\lambda _n \in \sigma(T)$ and the operator $\lambda_n I - T$ is lower semi-Fredholm for any $n \in \mathds{N}$.

      Applying in the first case Theorem 3.10 from~\cite{Ki1} and in the second case Theorem~\ref{t6} we can see that there is a closed subset $K_\infty$ of $K$ such that $\varphi(K_\infty) = K_\infty$ and $\rho(T, C(K_\infty)) = |\lambda|=\rho_{min}(T)$.

      We claim that the operator $\lambda I - T$ considered on the space $C(K_\infty)$ is lower semi-Fredholm. Indeed, otherwise there is a singular sequence $\mu_n \in C(K_\infty)^\prime$ such that $\|\mu_n\|=1$ and $T^\prime \mu_n - \lambda \mu_n \rightarrow 0$. Considering the measures $\mu_n$ as elements of $C(K)^\prime$ we come to a contradiction.

      Applying the punctured neighborhood theorem to the operator $\lambda I - T$ on $C(K_\infty)$ we see that the set $\sigma(T, C(K_\infty))$ is not rotation invariant. By Theorem 3.12 from ~\cite{Ki1} there is a $\varphi$-periodic point $k \in K_\infty$ such that $|w_p(k)| < |\lambda|^p$, where $p$ is the period of $k$. But the last inequality contradicts the definition of $\rho_{min}(T)$.
       \end{proof}

       To state our next result we have to recall the following definition introduced in~\cite{FK}

       \begin{definition} \label{d5}
         Let $K$ be a compact Hausdorff space. We say that $K \in AH$ if every almost homeomorphism of $K$ onto itself is a homeomorphism.
       \end{definition}

       \begin{theorem} \label{t8} Let $K$ be a compact Hausdorff space, $\varphi$ be a continuous non-invertible map of $K$ onto itself, and $w \in C(K)^{-1}$. Assume that the set of all eventually $\varphi$-periodic points is of first category in $K$.

       Assume that $K \in AH$. Then,
       \begin{equation*}
         \sigma_f(T) = \sigma(T).
       \end{equation*}
       \end{theorem}

       \begin{proof} Consider $\lambda \in \sigma(T)$. we have to consider three cases.
       \begin{enumerate}
         \item $|\lambda| < \rho_{min}(T)$. The operator $\lambda I - T$ cannot be Fredholm by Theorem~\ref{t4}.
         \item $|\lambda| = \rho_{min}(T)$ or $|\lambda| = \rho(T)$. The operator $\lambda I - T$ cannot be Fredholm by Corollary~\ref{c3}
         \item $\rho_min(T) < |\lambda| < \rho(T)$. The operator $\lambda I - T$ cannot be Fredholm by Theorem~\ref{t7}.
       \end{enumerate}
         \end{proof}

         While a complete characterization of the class AH remains unknown, numerous sufficient conditions guaranteeing that $K \in AH$ were obtained by Friedler and Kitover in~\cite{FK} and by Vermeer in~\cite{Ve}. We list some of this conditions in Corollary~\ref{c8}.

         \begin{corollary} \label{c8} Let $K$ be a compact Hausdorff space, $\varphi$ be a continuous non-invertible map of $K$ onto itself, and $w \in C(K)^{-1}$. Assume that the set of all $\varphi$-periodic points is of first category in $K$.

         Assume one of the following conditions.
         \begin{itemize}
           \item The compact space $K$ is extremally disconnected and has no isolated points.
           \item The compact space $K$ is an $F$-space without isolated points that satisfies the countable chain condition.
           \item The compact space $K$ is arcwise connected and one of the following conditions is satisfied
               \begin{enumerate}[(a)]
                 \item The fundamental group $\Pi_1(X)$ is finite.
                 \item The fundamental group $\Pi_1(X)$ is abelian.
                 \item The fundamental group $\Pi_1(X)$ is finitely generated.
               \end{enumerate}
           \item $K$ is a convex, compact subset of a linear topological space.
           \item $K$ is a compact, arcwise connected subset of the plane $\mathds{R}^2$ such that $\mathds{R}^2 \setminus K$ consists of a finite number of components.
           \item $K$ is a locally simply connected compact subset of $\mathds{R}^2$.
           \item $K = X \times Y$ where $X$ and $Y$ are locally connected compact spaces that have no isolated points.
           \end{itemize}

         Then $\sigma_f(T) = \sigma(T)$.
         \end{corollary}

\end{document}